\documentclass{amsart}
 \usepackage{epsfig}
\usepackage{graphicx}
\usepackage{amscd}
\usepackage{amsmath}
\usepackage{amsxtra}
\usepackage{amsfonts}
\usepackage{amssymb}
\usepackage[all,dvips,arc,curve,color,frame]{xy}

\newtheorem{theorem}{Theorem}[section]
\newtheorem{corollary}[theorem]{Corollary}
\newtheorem{lemma}[theorem]{Lemma}
\newtheorem{proposition}[theorem]{Proposition}

\theoremstyle{definition}
\newtheorem{definition}[theorem]{Definition}
\newtheorem{remark}[theorem]{Remark}
\newtheorem{observation}[theorem]{Observation}

\newtheorem{example}[theorem]{Example}
\theoremstyle{remark}

\renewcommand{\theclaim}{\textup{\theclaim}}

\numberwithin{equation}{section}

\def\openone

{\mathchoice

{\hbox{\upshape \small1\kern-3.3pt\normalsize1}}

{\hbox{\upshape \small1\kern-3.3pt\normalsize1}}

{\hbox{\upshape \tiny1\kern-2.3pt\SMALL1}}

{\hbox{\upshape \Tiny1\kern-2pt\tiny1}}}

\makeatletter

\newbox\ipbox

\newcommand{\ip}[2]{\left\langle #1\, , \,#2\right\rangle}
\newcommand{\diracb}[1]{\left\langle #1\mathrel{\mathchoice

{\setbox\ipbox=\hbox{$\displaystyle \left\langle\mathstrut
#1\right.$}

\vrule height\ht\ipbox width0.25pt depth\dp\ipbox}

{\setbox\ipbox=\hbox{$\textstyle \left\langle\mathstrut
#1\right.$}

\vrule height\ht\ipbox width0.25pt depth\dp\ipbox}

{\setbox\ipbox=\hbox{$\scriptstyle \left\langle\mathstrut
#1\right.$}

\vrule height\ht\ipbox width0.25pt depth\dp\ipbox}

{\setbox\ipbox=\hbox{$\scriptscriptstyle \left\langle\mathstrut
#1\right.$}

\vrule height\ht\ipbox width0.25pt depth\dp\ipbox}

}\right. }

\newcommand{\dirack}[1]{\left. \mathrel{\mathchoice

{\setbox\ipbox=\hbox{$\displaystyle \left.\mathstrut
#1\right\rangle$}

\vrule height\ht\ipbox width0.25pt depth\dp\ipbox}

{\setbox\ipbox=\hbox{$\textstyle \left.\mathstrut
#1\right\rangle$}

\vrule height\ht\ipbox width0.25pt depth\dp\ipbox}

{\setbox\ipbox=\hbox{$\scriptstyle \left.\mathstrut
#1\right\rangle$}

\vrule height\ht\ipbox width0.25pt depth\dp\ipbox}

{\setbox\ipbox=\hbox{$\scriptscriptstyle \left.\mathstrut
#1\right\rangle$}

\vrule height\ht\ipbox width0.25pt depth\dp\ipbox}

} #1\right\rangle}

\newcommand{\cj}[1]{\overline{#1}}

\newcommand{\bz}{\mathbb{Z}}

\newcommand{\br}{\mathbb{R}}
\newcommand{\bc}{\mathbb{C}}

\newcommand{\bn}{\mathbb{N}}

\def\blfootnote{\xdef\@thefnmark{}\@footnotetext}

\input xy
\xyoption{all}
\usepackage{amssymb}

\def\F{\mathcal{F}}

\def\H{\mathcal{H}}
\def\P{\mathcal{P}}
\def\LL{\mathcal{L}}

\def\-{^{-1}}

\def\D{\mathcal{D}}

\def\K{\mathcal{K}}
\def\ty{\emptyset}

\def\Fun{\operatorname*{Fun}}
\def\dom{\operatorname*{dom}}

\def\spn{\operatorname*{span}}
\def\V{\mathcal{V}}
\def\W{\mathcal{W}}

\def\Fin{\operatorname*{Fin}}

\begin{document}
\title[Spectral duality for unbounded operators]{Spectral duality for a class of unbounded operators}
\author{Dorin Ervin Dutkay}
\blfootnote{Research supported in part by a grant from the National Science Foundation DMS-0704191}
\address{[Dorin Ervin Dutkay] University of Central Florida\\
	Department of Mathematics\\
	4000 Central Florida Blvd.\\
	P.O. Box 161364\\
	Orlando, FL 32816-1364\\
U.S.A.\\} \email{ddutkay@mail.ucf.edu}

\author{Palle E.T. Jorgensen}
\address{[Palle E.T. Jorgensen]University of Iowa\\
Department of Mathematics\\
14 MacLean Hall\\
Iowa City, IA 52242-1419\\}\email{jorgen@math.uiowa.edu}
\thanks{} 
\subjclass[2000]{18A30, 31C20, 34L16, 34A45, 37A50, 46E22 , 47A75, 47B39}
\keywords{Spectrum, approximation, unbounded operator, reproducing kernel, discrete potentials, analysis on graphs, eigenvalue.}

\begin{abstract}
   We establish a spectral duality for certain unbounded operators in Hilbert space. The class of operators includes discrete graph Laplacians arising from infinite weighted graphs. The problem in this context is to establish a practical approximation of infinite models with suitable sequences of finite models which in turn allow (relatively) easy computations.

   Let $X$ be an infinite set and let $\H$ be a Hilbert space of functions on $X$ with inner product $\ip{\cdot}{\cdot}=\ip{\cdot}{\cdot}_{\H}$. We will be assuming that the Dirac masses $\delta_x$, for $x\in X$, are contained in $\H$. And we then define an associated operator $\Delta$ in $\H$ given by
$$(\Delta v)(x):=\ip{\delta_x}{v}_{\H}.$$
Similarly, for every finite subset $F\subset X$, we get an operator $\Delta_F$.

If $F_1\subset F_2\subset\dots$ is an ascending sequence of finite subsets such that $\cup_{k\in\bn}F_k=X$, we are interested in the following two problems:

(a) obtaining an approximation formula 
$$\lim_{k\rightarrow\infty}\Delta_{F_k}=\Delta;$$
and

(b) establish a computational spectral analysis for the truncated operators $\Delta_F$ in (a).

\end{abstract}
\maketitle \tableofcontents
\section{Introduction}\label{intr}

 The purpose of this paper is twofold: first to prove that certain linear operators associated with discrete reproducing kernel-Hilbert spaces exhibit spectral duality. This is motivated by more traditional Green's function techniques for second order elliptic differential operators. Secondly we explore applications of the duality theorem to discrete Laplace operators in weighted (infinite) graphs. In particular we show (for the discrete case) that the Green's function may be realized  as an infinite matrix with entries counting length of paths of edges in a graph.

 There has been a recent increase in the interplay between discrete analysis and various continuous limits. While each topic in its own right has been studied for generations, the interconnections are of a more recent vintage, and they in turn have inspired a multitude of exciting new research trends. The motivations for this are manifold, coming in part from  numerical analysis, but also more recently from analysis on fractals, see e.g., \cite{DuJo07a,DuJo07b, BSU08,Str05,Str06},  from stochastic processes, from potential theory, Dirichlet forms \cite{Saw97}, and discrete Laplacians on weighted graphs \cite{JoPe08, DuJo08, Fab06}. These topics interact with mathematical physics, see e.g.,  \cite{JoPe08,DuJo08,Pow76,OP96}, and with signal processing \cite{DuJo07a, DHP08, Jor83}. But independently of applications, the same themes have an operator theoretic dimension of interest in its own right, see e.g., \cite{AKLW07, Jor78} ; as well as spectral theory \cite{Jor81}. A common thread for this is the use of positive definite functions and reproducing kernel Hilbert spaces \cite{BCR84, Jor89, Jor90, Par70, ParSch72}.

  In a variety of studies, the authors have used special subclasses of reproducing kernel Hilbert spaces (RKHSs), and each case appears in isolation; for example, the authors of \cite{DuCu08} use RKHSs in a systematic study of Fredholm operators, \cite{FJKO05} in potential theory, \cite{GeGo05, GTHB05, HCDB07} in physics, \cite{HKK07, HCDB07, GTHB05} in signal processing, \cite{Pre07} in statistics, and \cite{SZ08, TT08} in harmonic analysis. One aim of the present paper is to unify these approaches.

   In this paper we take up two themes, one we call spectral reciprocity, and the other is a computational approximation scheme (sections \ref{trun} and \ref{gree}). Both themes interact with the various related developments covered in the above cited papers.

     The paper is organized as follows: In section \ref{hilb} we introduce the Hilbert spaces which admit spectral duality. Let $X$ be an infinite set, and let $\H$ be a Hilbert space of functions on $X$. The crucial restriction on the pair $X, \H$ is that the $\delta$ point masses are assumed to lie in the Hilbert space $\H$, (Definition \ref{def1_1}).

The distinction between the discrete and continuous models is illustrated with examples from the theory of stochastic processes. In section \ref{grap} we show that the framework of graph Laplacians is included in the setup. Section \ref{diag} offers a way of diagonalizing these operators. The idea is analogous to a method used by Karhunen-Loeve (see e.g., \cite{JoSo07}), but different in that it creates finite matrix approximations to the operator in a global ambient Hilbert space. In section \ref{trun} we study approximation: An ascending system of finite subsets in $X$ is chosen with union equal to $X$; and we then show that the corresponding sequence of finite truncations converges. The last theorem identifies a rigorous Green's function for graph Laplacians.

    Thus there are two interesting and interdisciplinary links to operators in symmetric Hilbert spaces (Definition 2.1). It is via operators in these Hilbert spaces built on infinite discrete spaces.

     Iterated function systems (abbreviated IFS, \cite{Hut81}) serve in two ways as a link between analysis on discrete systems on one side and operator theory on the other.  

    Recall that IFSs generate fractal images arising in numerous applications: For example, some IFS-fractals may be built as limits of iterated backwards trajectories of a dynamical system associated to a fixed endomorphism $T : X \rightarrow X$. The generation of the fractals is via recursive procedures applied to branches of a choice of inverse mappings for $T$.  As attractors, we then get limit fractal-sets and fractal measures $\mu$. So in this way the Hilbert space $L^2(\mu)$ arises as a limit of Hilbert spaces; starting with a graph and passing to the limit.

    On the discrete side, the graph $G$ has vertices $G^0$ and edges $G^1$. The first approach (see e.g., \cite{JoPe08}) is to model IFSs with infinite vertex sets $G^0$, and associated Hilbert spaces of functions on $G^0$. In the second approach (e.g., \cite{KeUr07}) one starts with an IFS, and then there is an associated graph $G$ with vertex $G^0$ set a singleton, but instead with edges made up of an infinite set $G^1$ of self-loops.

\section{Hilbert spaces of functions}\label{hilb}

We show that Hilbert spaces of functions which contain the corresponding $\delta$ point masses induce operators arising as graph Laplacians of weighted graphs.

The general setup in our paper is as follows: An infinite set $X$ is given, and we consider Hilbert spaces of functions on $X$. One of the Hilbert spaces will be simply $l^2(X)$. By this we mean the Hilbert space of all functions $u:X\rightarrow\bc$ such that 
\begin{equation}\label{eq1_1}
\sum_{x\in X}|u(x)|^2<\infty.
\end{equation}

If $u,v\in l^2(X)$, the inner product will be denoted
\begin{equation}\label{eq1_2}
\ip{u}{v}_2:=\sum_{x\in X}\cj u(x)v(x).
\end{equation}

Let $\F$ be the set of all finite subsets $F\subset X$. Then the expression in \eqref{eq1_1} is by definition 
\begin{equation}\label{eq1_3}
\sup_{F\in\F}\sum_{x\in F}|u(x)|^2.
\end{equation}
However because of applications, to be outlined later, for a fixed set $X$, it will be necessary for us to consider other Hilbert spaces $\H$ of functions on $X$.

\begin{definition}\label{def1_1}
Let $\H$ be a Hilbert space of functions on some set $X$. We say that $\H$ is {\it symmetric} if the Dirac functions $\delta_x$ are in $\H$, where
\begin{equation}\label{eq1_4}
\delta_x(y)=\left\{\begin{array}{cc}1&\mbox{ if } y=x\\
0&\mbox{ if }y\in X\setminus\{x\}.
\end{array}\right.
\end{equation}

\end{definition}

  For practical computations we offer in section \ref{diag} a method of finite reduction. As an application we give in Corollary \ref{cor4_5} a necessary and sufficient condition for a Hilbert space of functions to contain its Dirac delta-functions (Definition \ref{def1_1}).

We will primarily be interested in the case when the set $X$ is countably infinite; see especially section \ref{grap} below where we will take $X$ to be the set of vertices in a given weighted graph. Because of applications to electrical networks, see \cite{JoPe08} and the references cited there, every weighted graph comes with an associated Hilbert space $H_E$. In the applications, $H_E$ will denote a space of functions on the vertices of the graph, representing a voltage distribution; and, if $u\in H_E$, then $\|u\|_{H_E}^2$ will be the energy of the configuration represented by $u$.

The following example is different and applies to continuous models; for example models of stochastic processes.

\begin{example}\label{ex1_1.5}
Let $X=[0,1)$. We will be considering functions on $X$ modulo constants. Hence the constant function $1$ on $X$ will be identified with $0$. If $f$ is a function on $[0,1)$, the derivative $f'=\frac{d}{dx}f$ is understood in the sense of distributions. Set 
$$\H:=\{f\,|\, f'\in L^2(0,1)\},$$
\begin{equation}\label{eqA1}
\|f\|_{\H}^2:=\int_0^1|f'(x)|^2\,dx;\mbox{ and }
\end{equation}

\begin{equation}\label{eqA2}
\ip{f_1}{f_2}_{\H}:=\int_0^1\cj{f_1'(x)}f_2'(x)\,dx,\quad\mbox{ for }f_1,f_2\in\H.
\end{equation}

Note that if $f\in\H$, then $f'\in L^2$ and 
\begin{equation}\label{eqA3}
F(x):=\int_0^x f'(t)\,dt
\end{equation}
is well defined. Moreover, the derivative $\frac{d}{dx}F$ exists pointwise a.e. on $[0,1)$. As distributions, $\frac{d}{dx}F$ and $f'$ agree.

On $[0,1)$, consider the following family of functions $\{v_x\}$ indexed by $x\in X$. Set 
\begin{equation}\label{eqA4}
v_x(y)=\left\{\begin{array}{cc}
y&\mbox{ if } 0\leq y\leq x\\
0&\mbox{ if }y<0\\
x&\mbox{ if }x<y.
\end{array}\right.
\end{equation}

Writing in the sense of distributions we arrive at the following formula:

For every $f\in\H$,
\begin{equation}\label{eqA5}
\ip{v_x}{f}\stackrel{\mbox{by\eqref{eqA2}}}{=}\int_0^1v_x'(y)f'(y)\,dy\stackrel{\mbox{by\eqref{eqA4}}}{=}\int_0^xf'(y)\,dy\stackrel{\mbox{by\eqref{eqA3}}}{=}f(x).
\end{equation}
Hence $v_x\in\H$, and 
\begin{equation}\label{eqA6}
\ip{v_{x_1}}{v_{x_2}}=\min(x_1,x_2)=x_1\wedge x_2\mbox{ for all }x_1,x_2\in X.
\end{equation}

\begin{proposition}\label{propA1}
$\H$ is not a symmetric Hilbert space; i.e., if $x\in[0,1]$ then $\delta_x$ is not in $\H$.
\end{proposition}

\begin{proof}
The claim is that there is not a vector $f\in \H$ such that 
\begin{equation}\label{eqA7}
\varphi''(x)=\ip{f}{\varphi}_{\H}=\int_0^1\cj{f'(y)}\varphi'(y)\,dy
\end{equation}
for all twice differentiable functions $\varphi\in C^2[0,1)$. To see that \eqref{eqA7} is a restatement of $\delta_x\in\H$, note that $v_x''=-\delta_x$ holds in the sense of distributions. But note that \eqref{eqA7} implies that there is a finite constant $C_x$ such that 
$$|\varphi''(x)|^2\leq C_x\int_0^1|\varphi'(y)|^2\,dy,\quad(\varphi\in C^2);$$
which is clearly impossible.
\end{proof}
\end{example}
\begin{remark}\label{rem1_2}
(See Definition \ref{def1_1} the general case) The condition that $\delta_x$ is in $\H$ for all $x\in X$ does not imply that $l^2(X)$ is contained in $\H$. So there are symmetric Hilbert spaces which do not contain $l^2(X)$.
\end{remark}

\begin{definition}\label{def1_3}
If $X$ is given, and $\H$ is a symmetric Hilbert space, we set
\begin{equation}\label{eq1_5}
(\Delta v)(x):=\ip{\delta_x}{v}_{\H},\quad(v\in\H,x\in X).
\end{equation}

Let $\Fun(X)=$the vector space of all functions $X\rightarrow\bc$. Then $\Delta$ is a linear operator from $\H$ into $\Fun(X)$. 

We set
\begin{equation}\label{eq1_6}
\dom(\Delta)=\mbox{domain of }\Delta=\{v\in\H\,|\,\Delta v\in\H\},
\end{equation}
and we say that $\Delta$ is densely defined if $\dom(\Delta)$ is dense in $\H$.

Let $\Fin=\Fin(X)=$ all finite linear combinations of $\{\delta_x\,|\,x\in X\}$, i.e., all finitely supported functions on $X$.

\end{definition}

\begin{definition}\label{def1_4}
Let $X$ and $\H$ be as in the previous definition. A pair of functions: $X\ni x\mapsto v_x\in\H$ and 
$X\ni x\mapsto w_x\in \Fin$ is said to be a {\it dual pair} if 
\begin{equation}\label{eq1_7}
\ip{v_x}{u}_{\H}=\ip{w_x}{u}_2,\quad(x\in X,u\in\H)
\end{equation}
and if the linear span of $\{v_x\,|\,x\in X\}$ is dense in $\H$.

A dual pair is said to be {\it symmetric} iff
\begin{equation}\label{eq1_8}
w_x(y)=\cj{w_y(x)},\quad(x,y\in X)
\end{equation}
\end{definition}

\begin{theorem}\label{th1_5}
Let $X,\H$ be as above, and let $(v_x)_{x\in X}$, $(w_x)_{x\in X}$ be a dual pair. Let $\Delta$ be the operator defined in \eqref{eq1_6}, and set $\V:=\spn\{v_x\,|\,x\in X\}=$ all finite linear combinations.

Then $\V\subset\dom(\Delta)$ and $\Delta$ is Hermitian on its domain $\V$, i.e.,
\begin{equation}\label{eq1_9}
\ip{\Delta u}{v}_{\H}=\ip{u}{\Delta v}_{\H}\quad(u,v\in\V).
\end{equation}
Moreover
\begin{equation}\label{eq1_11}
\Delta v_x=w_x,\quad x\in X.
\end{equation}
\end{theorem}

\begin{proof}
We have for $x,y\in X$,
$$(\Delta v_x)(y)=\ip{\delta_y}{v_x}_{\H}\stackrel{\mbox{by \eqref{eq1_7}}}{=}\ip{\delta_y}{w_x}_2=w_x(y)$$
Thus $\Delta v_x=w_x\in\H$ so $v_x\in \dom(\Delta)$, and therefore $\V\subset\dom(\Delta)$.

If $u\in\dom(\Delta)$, then 
\begin{equation}\label{eq1_10}
\ip{v_x}{\Delta u}_{\H}=\ip{w_x}{u}_{\H},\quad(x\in X).
\end{equation}
Indeed,
$$\ip{v_x}{\Delta u}_{\H}=\ip{w_x}{\Delta u}_{2}=\sum_{y\in X}\cj{w_x(y)}\ip{\delta_y}{\Delta u}_2
\stackrel{\mbox{by \eqref{eq1_5}}}{=}\sum_{y\in X}\cj{w_x(y)}\ip{\delta_y}{u}_{\H}=\ip{w_x}{u}_{\H}.$$

So if $x_1,x_2\in X$, then 
$$\ip{v_{x_1}}{\Delta v_{x_2}}_{\H}=\ip{w_{x_1}}{v_{x_2}}_{\H}=\ip{w_{x_1}}{w_{x_2}}_2\stackrel{\mbox{by \eqref{eq1_7}}}{=}\ip{v_{x_1}}{w_{x_2}}_{\H}\stackrel{\mbox{by \eqref{eq1_10}}}{=}\ip{\Delta v_{x_1}}{v_{x_2}}_{\H}.$$

Since $\V=\spn\{v_x\,|\,x\in X\}$  the desired conclusion \eqref{eq1_9} holds.
\end{proof}

%
%

\section{Graph Laplacians}\label{grap}
We show that every weighted graph $G$ induces a Laplace operator and an energy Hilbert space of functions on the vertices of $G$; and moreover that this setup is included in that of section \ref{hilb}. This is then used in obtaining solutions to a potential theory problem on $G$.

\begin{definition}\label{def3_1} Weighted graph.

Let $G^0$ be a set. Let $G^1\subset G^0\times G^0$ be a subset such that $(x,x)\not\in G^1$ if $x\in G^0$.
\def\Nbh{\operatorname*{Nbh}}
For $x\in G^0$, set 
\begin{equation}\label{eq3_1}
\Nbh(x):=\{y\in G^0\,|\, (xy)\in G^1\}.
\end{equation}
We say that $(xy)$ is an edge if $(xy)\in G^1$; and the points in $G^0$ are called vertices. Further we shall use the notation 
$$(xy)\in G^1\Leftrightarrow x\sim y$$

Further assume 
\begin{equation}\label{eq3_1.5}
(xy)\in G^1\Leftrightarrow (yx)\in G^1
\end{equation}

Let $\mu:G^1\rightarrow\br_+$ be a function such that 
\begin{equation}\label{eq3_2}
\mu(x):=\sum_{y,y\sim x}\mu_{xy}<\infty,\mbox{ for all }x\in G^0.
\end{equation}
Further assume $\mu_{xy}=\mu_{yx}$ for all $(xy)\in G^1$.

We will assume that $G=(G^0,G^1)$ is {\it connected}, i.e., for every pair $x,y\in G^0$ there is a finite subset $\{e_0,e_1,\dots,e_n\}\subset G^1$, depending on $x$ and $y$ such that $e_i=(x_ix_{i+1})$ , $x_0=x$ and $x_{n+1}=y$.

\end{definition}

\begin{definition}\label{def3_2}
The energy Hilbert space $H_E$. For functions $u$ and $v$ on $G^0$, set
\begin{equation}\label{eq3_3}
\ip{u}{v}_E:=\frac12\sum_{x,y\, x\sim y}\mu_{xy}(\cj{u(x)}-\cj{u(y)})(v(x)-v(y)).
\end{equation}

More precisely, we will work with functions on $G^0$ modulo the constants. We say that $u\in H_E$ iff
\begin{equation}\label{eq3_4}
\|u\|_{H_E}^2:=\frac12\sum_{x,y\, x\sim y}\mu_{xy}|u(x)-u(y)|^2<\infty.
\end{equation}
\end{definition}

\begin{definition}\label{def3_3}
The graph Laplacian.  

Let $(G,\mu)$ be a weighted graph. We define the {\it graph Laplacian} $\Delta=\Delta_{(G,\mu)}$ initially on all functions on $G^0$ as follows
\begin{equation}\label{eq3_5}
(\Delta u)(x):=\sum_{y\sim x}\mu_{xy}(u(x)-u(y))=\mu(x)u(x)-\sum_{y\sim x}\mu_{xy}u(y).
\end{equation}
\end{definition}
In section \ref{hilb} we started with a symmetric Hilbert space (Definition \ref{def1_1}), and we derived an associated family of operators $\Delta$ from the Hilbert space setup. In this section, the point of view is reversed: we begin with a graph Laplacian $\Delta$ and an associated energy Hilbert space. It turns out that the class of operators in section 2 includes all the graph Laplacians. 

\begin{lemma}\label{lem3_4}
The energy Hilbert space $H_E$ associated with a weighted graph $(G,\mu)$ is symmetric, i.e., for all $x\in G^0$, we have $\delta_x\in H_E$. Moreover
\begin{equation}\label{eq3_6}
\|\delta_x\|_{H_E}^2=\mu(x);
\end{equation}
\begin{equation}\label{eq3_7}
\ip{\delta_x}{\delta_y}_E=\left\{\begin{array}{cc}
-\mu_{xy}&\mbox{ if }x\sim y\\
0&\mbox{ if } x\neq y\mbox{ and }(xy)\not\in G^1;
\end{array}
\right.
\end{equation}
and
\begin{equation}\label{eq3_8}
(\Delta u)(x)=\ip{\delta_x}{u}_E.
\end{equation}
\end{lemma}

\begin{proof}
Let $x_0\in G^0$. Then 
$$\|\delta_{x_0}\|_{H_E}^2=\frac12\sum_{x,y, x\sim y}\mu_{xy}(\delta_{x_0}(x)-\delta_{x_0}(y))^2=\frac12\left(\sum_{y\sim x_0}\mu_{x_0y}+\sum_{y\sim x_0}\mu_{yx_0}\right)\stackrel{\mbox{by \eqref{eq3_1.5}}}{=}\sum_{y\sim x_0}\mu_{x_0y}\stackrel{\mbox{by \eqref{eq3_2}}}{=}\mu(x_0)<\infty.$$

Let $(x_0y_0)\in G^1$.
$$\ip{\delta_{x_0}}{\delta_{y_0}}_E\stackrel{\mbox{by \eqref{eq3_3}}}{=}\frac12\sum_{x\sim y}\mu_{xy}(\delta_{x_0}(x)-\delta_{x_0}(y))(\delta_{y_0}(x)-\delta_{y_0}(y))=-\frac12(\mu_{x_0y_0}+\mu_{y_0x_0})\stackrel{\mbox{by \eqref{eq3_1.5}}}{=}-\mu_{x_0y_0}.$$

It is clear that $\ip{\delta_{x_0}}{\delta_{y_0}}=0$ if $x_0\neq y_0$ and $(x_0y_0)\not\in G^1$. 

We finally prove \eqref{eq3_8}. Let $x_0\in G^0$, and let $u\in H_E$. Then 
$$(\Delta u)(x_0)\stackrel{\mbox{by \eqref{eq3_5}}}{=}\sum_{y\sim x_0}\mu_{x_0y}(u(x_0)-u(y))=$$$$
\frac12\left(\sum_{y\sim x_0}\mu_{x_0y}(1-0)(u(x_0)-u(y))+\sum_{y\sim x_0}\mu_{yx_0}(0-1)(u(y)-u(x_0))\right)=$$
$$\frac12\sum_{x\sim y}\mu_{xy}(\delta_{x_0}(x)-\delta_{x_0}(y))(u(x)-u(y))\stackrel{\mbox{by \eqref{eq3_3}}}{=}\ip{\delta_{x_0}}{u}_E.$$
\end{proof}

\begin{theorem}\label{th3_5}
Let $(G,\mu)$ be a weighted graph; let $\Delta$ be the corresponding graph Laplacian, and let $H_E$ be the energy Hilbert space. Let $w:G^0\rightarrow \bc$ be a function on the vertices satisfying

(a) $w\in \Fin(=$ finite linear span of $\{\delta_x\,|\,x\in G^0\}$); \\
and

(b) $\sum_{x\in G^0}w_x=0$.

Then there is a $v\in H_E$ such that 
\begin{equation}\label{eq3_9}
\Delta v=w.
\end{equation}
\end{theorem}

\begin{remark}\label{rem3_6}
Before proving the theorem, we show by a simple example that neither of the two restrictions (a) or (b) on the function $w$ may be dropped. We will give examples when some function $w$ does not satisfy one of the two conditions. While there will always be a function $v:G^0\rightarrow\bc$ which satisfies \eqref{eq3_9}, the point is that none of the solutions $v$ will be in $H_E$, i.e., the solutions $v$ will have infinite energy, i.e., $\|v\|_{H_E}^2=\infty$.
\end{remark}

\begin{example}\label{ex3_4}
Let $(G,\mu)=(\bz,1)$. By this we mean that $G=(G^0,G^1)$ has 
\begin{equation}\label{eq3_10}
\left\{\begin{array}{c}G^0=\bz\\
G^1=\{(n,n+1)\,|\,n\in\bz\}\\
\mu_{(n,n+1)}=1,\quad(n\in\bz).
\end{array}\right.
\end{equation}

It follows from \eqref{eq3_5} that 
$$(\Delta u)(x)=2u(x)-u(x-1)-u(x+1),\quad(x\in\bz).$$
The following facts are from \cite{JoPe08}:

{\it Fact 1.} The only solutions $v$ to the equation 
\begin{equation}\label{eq3_11}
\Delta v=0\mbox{ on }\bz
\end{equation}
have the form $v(x)=Ax+a$, where $A$ and $a$, are constants.

{\it Fact 2.} On $\bz$ set
\begin{equation}\label{eq3_12}
v_+(x)=\left\{\begin{array}{cc}
x&\mbox{ if }x\geq0\\
0&\mbox{ if }x<0;
\end{array}\right.
\end{equation}
and 
\begin{equation}\label{eq3_13}
v_-(x)=\left\{\begin{array}{cc}
0&\mbox{ if }x>0\\
-x&\mbox{ if }x\leq0.
\end{array}\right.
\end{equation}
Then $v_{\pm}\not\in H_E$, i.e., $\|v_{\pm}\|_{H_E}^2=\infty$, and 
\begin{equation}\label{eq3_14}
\Delta v_+=\Delta v_-=-\delta_0.
\end{equation}

Combining the two facts, we see immediately that the equation 
\begin{equation}\label{eq3_15}
\Delta v=\delta_0
\end{equation}
has no solutions in $H_E$. Note that $\delta_0\in\Fin$, but does not satisfy condition (b) in the theorem.

The equation 
\begin{equation}\label{eq3_16}
\Delta u=v_+-v_-(=x)
\end{equation}
on $\bz$ does not have any solutions in $H_E$. Note that the function $v_+-v_-$ on the right hand side in \eqref{eq3_16} does satisfy (b), but $v_+-v_-$ is not in $\Fin$.
\end{example}

We now turn to the proof of Theorem \ref{th3_5}. The following lemma is helpful:

\begin{lemma}\label{lem3_7}
Let $(G,\mu)$ be a weighted graph, and let $\W$ denote the linear space of functions $w:G^0\rightarrow\bc$ satisfying conditions (a)-(b) in the statement of Theorem \ref{th3_5}. Then 
$$\W=\{w\in\Fin\,|\,\sum_{x\in G^0}w_x=0\}=\spn\{\delta_x-\delta_y\,|\,x,y\in G^0\}.$$
\end{lemma}

\begin{proof}
Induction on $\#\{x\,|\,w_x\neq0\}$.
\end{proof}

\begin{proof}[Proof of Theorem \ref{th3_5}]
By the lemma, it is enough to show that for any pair $x,y\in G^0$, $x\neq y$, the equation 
\begin{equation}\label{eq3_17}
\Delta v=\delta_x-\delta_y
\end{equation}
has a solution $v\in H_E$. 

Now fix $x$ and $y$ in $G^0$. Using Riesz' lemma, we first prove that there is a unique $v\in H_E$ such that
\begin{equation}\label{eq3_18}
\ip{v}{u}_{E}=u(x)-u(y),\quad(u\in H_E)
\end{equation}

Since $G$ is connected, there is a finite subset $\{e_0,\dots, e_n\}\subset G$ such that $e_i=(x_i,x_{i+1})\in G^1$, $x_0=x$ and $x_{n+1}=y$. Then 
$$|u(x)-u(y)|^2=\left|\sum_{i=0}^n(u(x_i)-u(x_{i+1}))\right|^2\leq
\sum_{i=0}^n\mu_{e_i}^{-1}\sum_{i=0}^n\mu_{e_i}|u(x_i)-u(x_{i+1})|^2
\stackrel{\mbox{by \eqref{eq3_4}}}{\leq}C_{xy}\|u\|_{H_E}^2.$$

Hence the existence of a solution $v$ in \eqref{eq3_18} follows from Riesz' lemma applied to $H_E$.

We note that $v$ satisfies \eqref{eq3_17}. Indeed, for all $z\in G^0$, we have 
$$(\Delta v)(z)\stackrel{\mbox{by \eqref{eq3_8}}}{=}\ip{\delta_z}{v}_E\stackrel{\mbox{by \eqref{eq3_18}}}{=}\delta_z(x)-\delta_z(y)=\delta_x(z)-\delta_y(z).$$
Hence the two sides in equation \eqref{eq3_17} agree as functions on $G^0$, and the proof is complete.

\end{proof}

\begin{definition}\label{def3_9}
{\it Positive semidefinite.} Let $X$ be a set. Set 
\begin{equation}\label{eq3_19}
\D:=\Fin=\mbox{ all finitely supported functions on }X=\{c:X\rightarrow\bc\,|\, \#\{x\in X\,|\, c_x\neq 0\}<\infty\}
\end{equation}

A function $M:X\times M\rightarrow\bc$ is said to be {\it positive semidefinite} iff
\begin{equation}
	\sum_{x,y}\cj{c_x}M(x,y)c_y\geq 0,\quad (x\in\D).
	\label{eqeq3_20}
\end{equation}
\end{definition}

\begin{theorem}\label{th3_10} (Parthasarathy-Schmidt \cite{ParSch72}.)

(a) Let $M:X\times X\rightarrow\bc$ be a function. Then the following conditions are equivalent:
\begin{equation}
	M\mbox{ is positive semidefinite.}
	\label{eq3_21}
\end{equation}
\begin{equation}
	\mbox{ There is a Hilbert space $\H$ and a function $v:X\rightarrow\H$ such that}
	\label{eq3_22}
\end{equation}
\begin{equation}
	M(x,y)=\ip{v_x}{v_y}_{\H},\quad(x,y\in\H).
	\label{eq3_23}
\end{equation}

(b) We say that two systems $v:X\rightarrow\H$, $v':X\rightarrow\H'$ in (a) are {\it unitarily equivalent} if there is a unitary isomorphism $W:\H\rightarrow\H'$ such that
\begin{equation}
	Wv_x=v_x',\quad(x\in X).
	\label{eq3_24}
\end{equation}

(c) If $v:X\rightarrow\H$ and $v':X\rightarrow\H'$ are two systems both satisfying \eqref{eq3_23} then $v$ and $v'$ are unitarily equivalent iff
\begin{equation}
	\cj\spn\{v_x\,|\,x\in X\}=\H,\mbox{ and }\,\cj\spn\{v_x'\,|\,x\in X\}=\H'.
	\label{eq3_25}
\end{equation}

\end{theorem}

\begin{corollary}\label{cor3_11}
Let $(G,\mu)$ be a weighted graph satisfying the conditions in Theorem \ref{th3_5}. Let $H_E$ be the energy Hilbert space and $\Delta$ the graph Laplacian.

(a) For every $x,y\in G^0$ let $v_{xy}$ be the unique solution in $H_E$ to equation \eqref{eq3_17}. Then for a fixed $y$, the function $G^0\times G^0\rightarrow\bc$, $(x_1,x_2)\mapsto\ip{v_{x_1y}}{v_{x_2y}}_E$ is positive semidefinite. Moreover, the function $(G^0\times G^0)\times(G^0\times G^0)\rightarrow\bc$, $(x_1y_1, x_2y_2)\mapsto\ip{v_{x_1y_1}}{v_{x_2y_2}}_E$ is positive semidefinite.

(b) Let $G=(G^0,G^1)$ be as in (a), and let $\mu:G^1\rightarrow\br_+$ be a function satisfying the conditions in Definition \ref{def3_1}. Let $M=M_\mu: G^0\times G^0\rightarrow\bc$ be
$$M(x,y)=\left\{\begin{array}{cc}
\mu(x)&\mbox{ if }x=y\\
-\mu_{xy}&\mbox{ if }(xy)\in G^1\\
0&\mbox{ if }x\neq y\mbox{ and }(xy)\not\in G^1.
\end{array}\right.$$
Then $M_\mu$ is positive semidefinite.
\end{corollary}

\section{Diagonalizing subsystems}\label{diag}

  It is known that positive semidefinite functions define reproducing kernel Hilbert spaces. In this section we identify which of these Hilbert spaces are symmetric (Definition\ref{def1_1}). And we solve the problem of diagonalizing finite subsystems.

Let $X$ be a set, and let $M:X\times X\rightarrow\bc$ be a positive semidefinite function. We will consider solutions $v:X\rightarrow\H$ to condition \eqref{eq3_23}, i.e., 
\begin{equation}
	M(x,y)=\ip{v_x}{v_y}_{\H},\quad(x,y\in X)
	\label{eq4_1}
\end{equation}
The next result shows that when restricting to finite subsystems, $\{v_x\,|\, x\in F\}$, $F\subset X$ finite, we may assume that the set $(v_x)_{x\in F}$ is linearly independent in $\H$.

\begin{definition}\label{def4_1}
Let $M:X\times X\rightarrow\bc$ be positive semidefinite. Let $\LL$ be the space of all finite linear combinations
\begin{equation}
	f_c(\cdot)=\sum_{x\in X}c_x M(\cdot,x)
	\label{eq4_2}
\end{equation}
Set
\begin{equation}
	\ip{f_a}{f_b}_{\H}:=\sum_{x,y}\cj{a_x}M(x,y)b_y\mbox{ for }f_a,f_b\in\LL.
	\label{eq4_3}
\end{equation}
Set
\begin{equation}
	\K:=\{f_c\in\LL\,|\,\sum_{x,y}\cj{c_x}M(x,y)c_y=0\},
	\label{eq4_4}
\end{equation}
the kernel of $M$.

Now set
\begin{equation}
	\LL\rightarrow\LL/\K\rightarrow\mbox{ Hilbert completion}=:\H_M,
	\label{eq4_5}
\end{equation}
\end{definition}

Set
$$v_x:=M(\cdot,x)\rightarrow\mbox{class} M(\cdot,x)\in\H_M.$$
Then $v_x=f_{\delta_x}$, i.e., $v_x=f_c$ with $c=\delta_x$; and 
\begin{equation}
	\ip{v_x}{f}=f(x),\quad(f\in\H_M).
	\label{eq4_6}
\end{equation}
Indeed
$$\ip{v_{x_0}}{f_c}=\sum_{x,y}\cj{\delta_{x_0}}(x)M(x,y) c(y)=\sum_y M(x_0,y) c(y)=f_c(x_0).$$
 We refer to \cite{Ar50} for the general theory of reproducing kernels.

 In the analysis below, the idea is to select finite subsets $F$ of a fixed ambient infinite set $X$; and it is assumed that $\H$ is a symmetric Hilbert space of functions on $X$. This method of finite reduction is motivated by computations, in that infinite sequences do not admit representations in computer registers.

\begin{lemma}\label{lem4_2}
Let $M:X\times X\rightarrow\bc$ be a positive semidefinite function, and let $\H_M$ be the Hilbert space in Definition \ref{def4_1}. Let $F\subset X$ be a finite subset, and let $M_F$ be the $\#F\times\#F$ matrix
\begin{equation}
	(M(x,y))_{x,y\in F}.
	\label{eq4_7}
\end{equation}
Then if $0$ is in the spectrum of $M_F$ with eigenvector $(c_x)_{x\in F}$, then $f_c$ in \eqref{eq4_2} represents the zero vector in $\H_M$.
\end{lemma}

\begin{proof}
Follows from Definition \ref{def4_1} and \eqref{eq4_5}: if $(c_x)_{x\in F}$ is an eigenvector for $M_F$ with eigenvalue $0$, i.e., 
$ M_F(c_x)_{X_\in F}=$ then 
$\sum_{y\in X}M(x,y)c_y=0$ for all $x\in F$ so 
$$\sum_{x,y\in F}\cj{c_x}M(x,y)c_y=0,\mbox{ i.e., }\|f_c\|_{\H_M}^2=0.$$ 

\end{proof}

\begin{remark}\label{rem4_2.5}
Since every positive semidefinite function $M$ induces a reproducing kernel Hilbert space $M\rightarrow \H_M$ via Definition \ref{def4_1}, it is important to note that the class of Hilbert spaces in Definition \ref{def1_1} are restricted in two ways: a symmetric Hilbert space $\H$ is a space {\it of functions} on a given set $X$ and $\delta_x\in\H$ for all $x\in X$. 

The following example shows that $\H_M$ may be obtained from a positive semidefinite function $M:X\times X\rightarrow\bc$, even though $\H_M$ is not a space of functions on $X$. 
\end{remark}
\begin{example}\label{ex4_2.5}\cite{AlKa07,AlLe08,JoOl00,Jor02} Let $X=[0,1)$, and set $M(x,y)=\frac{1}{1-xy}$. Then $M$ is positive semidefinite on $[0,1)$. Moreover the resulting Hilbert space (Definition \ref{def4_1}) $\H_M$ contains $\delta_x$ for all $x\in[0,1)$.

Let $u$ and $v$ be compactly supported distributions, and $u\otimes v$ the tensor product $(u\otimes v)(x,y):=u(x)v(y)$
where the right-hand side is evaluation on $C^\infty(\br^2)$, written $\ip{u(\cdot)v(\cdot)}{\psi(\cdot,\cdot)}$, $\psi\in C^\infty(\br^2)$. The $\H_M$-inner product is defined by
$$\ip{u}{v}_{\H_M}:=\ip{\cj{u}\otimes v}{\frac{1}{1-xy}}$$
where the right-hand side now denotes application of the distribution $\cj u\otimes v$ to $\psi(x,y)=\frac{1}{1-xy}$. 

If $\delta_0^{(n)}=\left(\frac{d}{dx}\right)^n\delta_0$, $n\in\bn_0$, are the distribution derivatives, then 
\begin{equation}
	u_n:=\frac{(-1)^n}{n!}\delta_0^{(n)},\quad(n\in\bn^0)
	\label{eq4_2.5}
	\end{equation}
is an orthonormal basis in $\H_M$. Indeed, if $\varphi\in C^\infty(-\epsilon,1)$ for some $\epsilon\in\br_+$ then $\varphi\in\H_M$, and the $\{u_n\}$ expansion in $\H_M$ is as follows
$$\varphi=\sum_{n=0}^\infty\ip{u_n}{\varphi}_{\H_M}u_n;$$
and for $x\in(0,1)$, we have 
$$\varphi(x)=\ip{\delta_x}{\varphi}_{\H_M}=\sum_{n=0}^\infty\frac{x^n}{n!}\varphi^{(n)}(0),$$
i.e., the Taylor expansion.
\end{example}
\begin{remark}\label{rem4_3}
Because of Lemma \ref{lem4_2}, we will assume in the sequel that when $M$ and $M_F$ are as described then $0$ is not in $\mbox{spec}_{l^2}(M_F)$.
\end{remark}

\begin{theorem}\label{th4_4}
Let $M:X\times X\rightarrow\bc$ be a positive semidefinite function, and let $\H_M$ be the Hilbert space in Definition \ref{def4_1}. Let $F\subset X$ be a finite subset, and set 
\begin{equation}
	\Lambda_F:={\operatorname*{spectrum}}_{l^2(F)}M_F
	\label{eq4_8}
\end{equation}
and
\begin{equation}
	\H_M(F)=\spn_{\H_M}\{v_x\,|\,x\in F\}.
	\label{eq4_9}
\end{equation}

Let $(\xi_\lambda)_{\lambda\in\Lambda_F}$ be an ONB in $l^2(F)$ satisfying
\begin{equation}
	M_F\xi_\lambda=\lambda\xi_\lambda\mbox{ on }F.
	\label{eq4_10}
\end{equation}
For $\lambda\in\Lambda_F$, set 
\begin{equation}
	u_\lambda(\cdot):=\frac{1}{\sqrt{\lambda}}\sum_{x\in F}\xi_\lambda(x)v_x(\cdot).
	\label{eq4_11}
\end{equation}
Then $(u_\lambda)_{\lambda\in\Lambda_F}$ is an ONB in $\H_M(F)$, and 
\begin{equation}
	v_x=\sum_{\lambda\in\Lambda_F}\sqrt{\lambda}\,\cj{\xi_\lambda}(x)u_\lambda\mbox{ for all } x\in F.
	\label{eq4_12}
\end{equation}
\end{theorem}

\begin{proof}
We first show that the system $(u_\lambda)_{\lambda\in\Lambda_F}$ in \eqref{eq4_11} is orthonormal in $\H_M$. Let $\lambda,\lambda'\in\Lambda_F$. Then 
$$\ip{u_\lambda}{u_{\lambda'}}_{\H_M}=\frac{1}{\sqrt{\lambda\lambda'}}\sum_{x,y\in F}\cj{\xi_\lambda(x)}\xi_{\lambda'}(y)\ip{v_x}{v_y}=
\frac{1}{\sqrt{\lambda\lambda'}}\sum_{x\in F}\cj{\xi_\lambda(x)}(M_F\xi_{\lambda'})(x)=$$

$$\stackrel{\mbox{by \eqref{eq4_10}}}{=}\frac{1}{\sqrt{\lambda\lambda'}}\sum_{x\in F}\cj{\xi_\lambda(x)}\lambda'\xi_{\lambda'}(x)=\sqrt{\frac{\lambda'}{\lambda}}\sum_{x\in F}\cj{\xi_\lambda(x)}\xi_{\lambda'}(x)=\sqrt{\frac{\lambda'}{\lambda}}\ip{\xi_\lambda}{\xi_{\lambda'}}_{l^2(F)}=\sqrt{\frac{\lambda'}{\lambda}}\delta_{\lambda,\lambda'}=\delta_{\lambda,\lambda'}.$$

By Lemma \ref{lem4_2} we see that $(u_\lambda)_{\lambda\in\Lambda_F}$ is indeed an ONB for $\H_M(F)$ and that
\begin{equation}
	P_F:=\sum_{\lambda\in\Lambda_F} |u_\lambda\rangle\langle u_\lambda|
	\label{eq4_13}
\end{equation}
is the orthogonal projection onto $\H_M(F)$. Note that we use Dirac's ``ket-bra'' notation on the right hand side of \eqref{eq4_13}. 

We now prove \eqref{eq4_12}. For $x\in F$ we have 

$$v_x=P_Fv_x=\stackrel{\mbox{by \eqref{eq4_13}}}{=}\sum_{\lambda\in\Lambda_F}\ip{u_\lambda}{v_x}u_{\lambda}=\sum_{\lambda\in\Lambda_F}\frac{1}{\sqrt{\lambda}}\sum_{y\in F}\cj{\xi_\lambda(y)}\ip{v_y}{v_x}u_{\lambda}=$$$$\stackrel{\mbox{by \eqref{eq4_10}}}{=}\sum_{\lambda\in\Lambda_F}\frac{1}{\sqrt{\lambda}}\lambda\cj{\xi_\lambda(x)}u_\lambda=\sum_{\lambda\in\Lambda_F}\sqrt{\lambda}\,\,\cj{\xi_\lambda(x)}u_\lambda$$
which is the desired formula \eqref{eq4_12}.
\end{proof}

\begin{corollary}\label{cor4_5}
Let $M:X\times X\rightarrow\bc$ be a positive semidefinite function, and let $\H_M$ be the Hilbert space in Definition \ref{def4_1}. Choose the system $\{v_x\}_{x\in X}\subset\H_M$ as in \eqref{eq4_5}-\eqref{eq4_6}. For every finite subset $F\subset X$, let 
\begin{equation}
	(\xi_\lambda^F(x))_{\lambda\in\Lambda_F,x\in F}
	\label{eq4_13.5}
\end{equation}
be the unitary $\#F\times\#F$ matrix from the construction in Theorem \ref{th4_4}. Then $\H_M$ is a symmetric Hilbert space (Definition \ref{def1_1}) iff
\begin{equation}
	\sup_{F\in\F}\sum_{\lambda\in\Lambda_F}\frac{|\xi_\lambda^F(x)|^2}{\lambda}<\infty.
	\label{eq4_14}
\end{equation}
\end{corollary}

\begin{proof}
Recall $\H_M$ is a symmetric Hilbert space iff $\delta_x\in\H_M$ for all $x\in X$. Assume this condition holds; and let $F\in\F=$the set of all finite subsets of $X$, and let $x_0\in X$.

From \eqref{eq4_13}, recall the formula for the projection onto $\H_M(F)$:
$$P_F=\sum_{\lambda\in\Lambda_F}|u_\lambda^F\rangle\langle u_\lambda^F|.$$
Since $\delta_{x_0}\in\H_M$, we have 
$$P_F\delta_{x_0}=\sum_{\lambda\in\Lambda_F}\ip{u_\lambda^F}{\delta_{x_0}}u_\lambda^F=\sum_{\lambda\in\Lambda_F}\frac{1}{\sqrt{\lambda}}\sum_{x\in F}\cj{\xi_\lambda^F(x)}\ip{v_x}{\delta_{x_0}}_{\H_M}u_\lambda^F=$$
$$\stackrel{\mbox{by \eqref{eq4_6}}}{=}\sum_{\lambda\in\Lambda_F}\frac{1}{\sqrt{\lambda}}\sum_{x\in F}\cj{\xi_\lambda^F(x)}\delta_{x_0}(x)u_\lambda^F=\sum_{\lambda\in\Lambda_F}\frac{1}{\sqrt{\lambda}}\cj{\xi_\lambda^F(x_0)}u_\lambda^F.$$
Since $(u_\lambda^F)_{\lambda\in\Lambda_F}$ is an ONB in $\H_M(F)$ by the theorem, we get 
$$\|P_F\delta_{x_0}\|_{\H_M}^2=\sum_{\lambda\in\Lambda_F}\frac{|\xi_\lambda^F(x_0)|^2}{\lambda}\leq\|\delta_{x_0}\|^2_{\H_M}<\infty.$$
Taking supremum over $\F$, the desired conclusion \eqref{eq4_14} now follows.

 Conversely, suppose \eqref{eq4_14} is satisfied for some vertex $x_0$. To prove that $\delta_{x_0}\in\H$, we shall need the following observations which may be of independent interest.
 \begin{observation}\label{obs1}
 Let $F$ and $F'$ be two finite sets, and assume $F\subset F'$. Then $\H_M(F)\subset\H_M(F')$; see \eqref{eq4_9}; and therefore
 \begin{equation}
	P_F\subset P_{F'},
	\label{eqo1}
\end{equation}
or equivalently
\begin{equation}
	P_F=P_{F'}P_F=P_FP_{F'}.
	\label{eqo2}
\end{equation}
For the corresponding two eigenvalue sets $\Lambda_F$ and $\Lambda_{F'}$ in \eqref{eq4_8} we have
\begin{equation}
	\min\{\lambda'\in\Lambda_{F'}\}\leq\min\{\lambda\in\Lambda_F\};
	\label{eqo3}
\end{equation}
and
\begin{equation}
	\max\{\lambda\in\Lambda_F\}\leq\max\{\lambda'\in\Lambda_{F'}\}.
	\label{eqo4}
\end{equation}
(Note that \eqref{eqo3}-\eqref{eqo4} follow from the min-max principle in spectral theory.)
 \end{observation}
 
\begin{observation}\label{obs2}
If $w\in\H_M$, and $F\subset F'$, then 
\begin{equation}
	\|P_Fw\|_\H^2\leq\|P_{F'}w\|_\H^2\leq\|w\|_\H^2,
	\label{eqo5}
\end{equation}
and
\begin{equation}
	\|P_Fw\|_\H^2=\sum_{\lambda\in\Lambda_F}|\ip{u_\lambda^F}{w}_\H|^2;
	\label{eqo6}
\end{equation}
and
$$|\ip{u_\lambda^F}{w}_\H|^2=\frac{1}{\lambda}\left|\sum_{x\in F}\cj{\xi_{\lambda}^F(x)}\ip{v_x}{w}_\H\right|^2=\frac{1}{\lambda}\left|\sum_{x\in F}\cj{\xi_\lambda^F(x)}(w(x)-w(0))\right|^2.$$
(In the application above, we used this principle to $w=\delta_{x_0}$) The proof details here are based on Theorem \ref{th4_4} Part I. 
\end{observation} 
 
 \begin{observation}\label{obs3}
 Suppose there is a $w\in\H_M$ such that
 \begin{equation}
	\ip{u_\lambda^F}{w}_\H=\frac{1}{\sqrt\lambda}\cj{\xi_\lambda^F(x_0)}
	\label{eqo7}
\end{equation}
for all $F\in\F$, and all $\lambda\in\Lambda_F$; then $w=\delta_{x_0}$.
 \end{observation}
 
 \begin{observation}\label{obs4}
 Assume \eqref{eq4_14}; then for all $F\in\F$, there exist some vector $\delta_{x_0}^F\in\H_M(F)$ such that
 \begin{equation}
	\ip{u_\lambda^F}{\delta_{x_0}^F}_\H=\frac{1}{\sqrt\lambda}\cj{\xi_\lambda^F(x_0)},\quad(\lambda\in\Lambda_F).
	\label{eqo8}
\end{equation}
 \end{observation}
 \begin{observation}\label{obs5}
 For every $(F_k)\subset\F$, $F_1\subset F_2\subset\dots$ such that $\cup_k F_k=X$, we have
 \begin{equation}
	\lim_{k\rightarrow\infty}\sum_{\lambda\in\Lambda_{F_k}}\frac1\lambda|\xi_\lambda^{F_k}(x_0)|^2=\sup_{F\in\F}\sum_{\lambda\in\Lambda_F}\frac{|\xi_\lambda^F(x)|^2}{\lambda}\mbox{ (see \eqref{eq4_14})}.
	\label{eqo9}
\end{equation}
 \end{observation}
 \begin{observation}\label{obs6}
 Let $(F_k)_{k\in\bn}$ be a system in $\F$ as in Observation \ref{obs5}; and choose vectors $\delta_{x_0}^{F_k}\in\H_M$ as in Observation \ref{obs4}. Then 
 $$\lim_{k,l\rightarrow\infty}\|\delta_{x_0}^{F_k}-\delta_{x_0}^{F_l}\|_{\H}=0,$$
 and so there exists a unique $w_{x_0}\in\H$ such that 
 $$\lim_{k\rightarrow\infty}\|\delta_{x_0}^{F_k}-w_{x_0}\|_\H=0.$$
 \end{observation}
 \begin{observation}\label{obs7}
 An application of Observation \ref{obs3} shows that $w_{x_0}=\delta_{x_0}$; i.e., that as functions on $X$, $w_{x_0}$ and $\delta_{x_0}$ coincide.

 To see this, note that the existence of $w_{x_0}\in\H$ is from Observation \ref{obs6}; and that its properties follow from a combination of all the preceding observations.
 \end{observation}
\end{proof}

We conclude this section with two examples: both will be needed later, both are discrete analogues of Example \ref{ex1_1.5}; and both yield energy Hilbert spaces $H_E=\H_M$ which are symmetric Hilbert spaces. This means that condition \eqref{eq4_14} of Corollary \ref{cor4_5} is satisfied in both examples.

\begin{example}\label{ex4_6}
(Example \ref{ex3_4} revisited)  As in Example \ref{ex3_4}, we take $(G,\mu)=(\bz,1)$; i.e., the graph with vertices $G^0=\bz$, and edges represented by nearest neighbors. 

Th argument in Example \ref{ex3_4} shows that for each $x\in G^0=\bz$, the equation 
\begin{equation}
	\Delta v_x=\delta_x-\delta_0
	\label{eqeq4_16}
\end{equation}
has a unique solution $v_x\in H_E$, and the graph of $v_x$ is represented in Figure \ref{fig1}; for the cases $x\in\bz_+$, $x\in\bz_-$ respectively. 
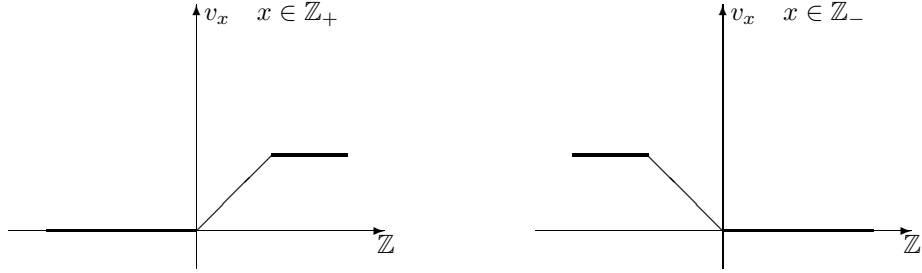
\begin{figure}[ht]
\vspace{2.5cm}
\centerline{
\vbox{
\hbox{
\setlength{\unitlength}{1cm}
\begin{picture}(1,1)
\linethickness{0.075mm}
\put(-6,0){\vector(1,0){5}}
\put(-3.5,-0.5){\vector(0,1){3.5}}
\linethickness{0.3mm}
\put(-5.5,0){\line(1,0){2}}
\linethickness{0.3mm}
\put(-3.5,0){\line(1,1){1}}
\linethickness{0.3mm}
\put(-2.5,1){\line(1,0){1}}
\put(-3.4,2.8){$v_x\quad x\in\bz_+$}
\put(-1.1,-0.3){$\bz$}
\linethickness{0.075mm}
\put(1,0){\vector(1,0){5}}
\put(3.5,-0.5){\vector(0,1){3.5}}
\linethickness{0.3mm}
\put(5.5,0){\line(-1,0){2}}
\linethickness{0.3mm}
\put(3.5,0){\line(-1,1){1}}
\linethickness{0.3mm}
\put(2.5,1){\line(-1,0){1}}
\put(3.6,2.8){$v_x\quad x\in\bz_-$}
\put(5.9,-0.3){$\bz$}
\end{picture}
}
}
}\caption{$v_x$ for $x\in\bz_+$ and for $x\in\bz_-$}\label{fig1}
\end{figure}

If $x\in\bz_+$ then 
\begin{equation}
	v_x(y)=\left\{\begin{array}{cc}
	0&\mbox{ if }y\leq0\\
	y&\mbox{ if }0\leq y\leq x\\
	x&\mbox{ if }x<y.
	\end{array}\right.
	\label{eq4_17}
\end{equation}
If $x\in\bz_-$, then 
\begin{equation}
	v_x(y)=\left\{\begin{array}{cc}
	-x&\mbox{ if }y<x\\
	-y&\mbox{ if }x\leq y\leq 0\\
	0&\mbox{ if }0\leq y.
	\end{array}\right.
	\label{eq4_18}
\end{equation}
An application of \eqref{eq3_3} in Definition \ref{def3_2} now yields
\begin{equation}
	\ip{v_{x_1}}{v_{x_2}}_E=\left\{\begin{array}{l}
	\min(x_1,x_2)=x_1\wedge x_2,\mbox{ if }x_1,x_2\in\bz_+\\
	|x_1|\wedge |x_2|\mbox{ if }x_1,x_2\in\bz_-\\
	0\mbox{ if }x_1\in\bz_+\mbox{ and }x_2\in\bz_-.
	\end{array}\right.
	\label{eq4_19}
\end{equation}
Hence a typical submatrix $M_F$ constructed by restriction to $F\times F$ from 
\begin{equation}
	M(x,y)=\ip{v_x}{v_y}_E
	\label{eq4_20}
\end{equation}
see Lemma \ref{lem4_2}, has the form
\begin{equation}
	\left(
	\begin{array}{cccccc}
	1&1&1&\cdots&1&1\\
	1&2&2&\cdots&2&2\\
	1&2&3&\cdots&3&3\\
	\vdots&\vdots&\vdots&\ddots&\vdots&\vdots\\
	1&2&3&\cdots&n-1&n-1\\
	1&2&3&\cdots&n&n
	\end{array}
	\right)
	\label{eq4_21}
\end{equation}
or a submatrix thereof.
\end{example}

\begin{example}\label{ex4_7}(\cite{DuJo08,JoPe08})
We take $(G,\mu)=(\mbox{tree},1)$; i.e., the graph with vertices $G^0=$the dyadic tree, see Figure \ref{fig3}.
\begin{figure}[ht]
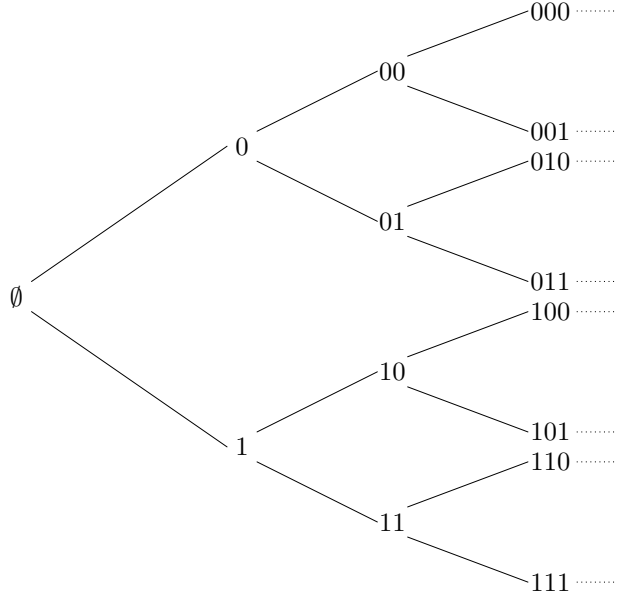

\[
\xy
(-78,2); (-52,20)**@{-};
(-80,0)*{\ty};
(-78,-2);(-52,-20)**@{-};
(-50,20)*{0};
(-50,-20)*{1};
(-48,22);(-32,30)**@{-};
(-48,18);(-32,10)**@{-};
(-48,-18);(-32,-10)**@{-};
(-48,-22);(-32,-30)**@{-};
(-30,30)*{00};
(-30,10)*{01};
(-30,-10)*{10};
(-30,-30)*{11};
(-28,32);(-12,38)**@{-};
(-28,28);(-12,22)**@{-};
(-28,12);(-12,18)**@{-};
(-28,8);(-12,2)**@{-};
(-28,-8);(-12,-2)**@{-};
(-28,-12);(-12,-18)**@{-};
(-28,-28);(-12,-22)**@{-};
(-28,-32);(-12,-38)**@{-};
(-9,38)*{000};
(-9,22)*{001};
(-9,18)*{010};
(-9,2)*{011};
(-9,-38)*{111};
(-9,-22)*{110};
(-9,-18)*{101};
(-9,-2)*{100};
(-6,38);(0,38)**@{.};
(-6,22);(0,22)**@{.};
(-6,18);(0,18)**@{.};
(-6,2);(0,2)**@{.};
(-6,-2);(0,-2)**@{.};
(-6,-18);(0,-18)**@{.};
(-6,-22);(0,-22)**@{.};
(-6,-38);(0,-38)**@{.};
\endxy
\]
\caption{$G_0=$ the dyadic tree.}\label{fig3}
\end{figure}
The empty word $\ty$ has two neighbors $0$ and $1$; and all other finite words $x=(\omega_1\omega_2\dots\omega_k)$, $\omega_i\in\{0,1\}$ have three neighbors 
\begin{equation}
	(\omega_1\omega_2\dots\omega_{k-1}),(\omega_1\omega_2\dots\omega_k0)\mbox{ and }(\omega_1\omega_2\dots\omega_k1)
	\label{eq4_22}
\end{equation}
written $x^*, (x0)$ and $(x1)$.

With $\mu\equiv1$, the Laplace operator $\Delta$ is (see \eqref{eq3_5})
$$(\Delta u)(\ty)=2u(\ty)-u(0)-u(1),$$
and 
\begin{equation}
	(\Delta u)(x)=3 u(x)-u(x^*)-u(x0)-u(x1).
	\label{eq4_23}
\end{equation}
For $x\in G^0\setminus\{\ty\}$, the equation
\begin{equation}
	\ip{v_x}{u}_E=u(x)-u(\ty)
	\label{eq4_24}
\end{equation}
has the unique solution $v_x\in H_E$ given as follows: There is a unique path $\P(x)$ of edges leading from $\ty$ to $x$: $(\ty,\omega_1),(\omega_1,\omega_1\omega_2),\dots,(\omega_1\dots\omega_{k-2},x^*),(x^*,x)$; see Figure \ref{fig4}. Then $v_x(y)=$the length of the path common to $\P(x)$ and $\P(y)$, so 
$$v_x(y)=\#(\P(x)\cap \P(y)).$$ 
\begin{figure}[ht]
\vspace{2cm}
\setlength{\unitlength}{1mm}
\begin{picture}(10,10)
\linethickness{0.075mm}
\put(-30,0){\line(1,1){10}}
\put(-30,0){$\bullet$}
\put(-33,0){$\ty$}
\put(-20,10){\line(1,-1){10}}
\put(-21,9){$\bullet$}
\put(-21,12){$\omega_1$}
\put(-10,0){\line(1,1){20}}
\put(-11,-1){$\bullet$}
\put(-11,-4){$\omega_1\omega_2$}
\put(-1,9){$\bullet$}
\put(1,6){$\omega_1\omega_2\omega_3$}
\put(9,19){$\bullet$}
\linethickness{0.075mm}
\multiput(10,18)(1,-1){10}{$\cdot$}
\put(20,8){$\bullet$}
\put(20,5){$\omega_1\dots\omega_{k-2}$}
\put(21,9){\line(2,1){10}}
\put(30,13){$\bullet$}
\put(30,15){$x^*$}
\put(31,14){\line(2,-1){10}}
\put(40,8){$\bullet$}
\put(40,11){$x$}
\end{picture}
\caption{$\P(x)$}\label{fig4}
\end{figure}
For the positive definite function $M$ in \eqref{eq4_1} we now get
\begin{equation}
	M(x,y)=\ip{v_x}{v_y}_E=\#(\P(x)\cap \P(y)),\quad(x,y\in G^0\setminus\{\ty\});
	\label{eq4_25}
\end{equation}
i.e., the length of the path common to $\P(x)$ and $\P(y)$.

Hence a typical submatrix $M_F$ constructed from \eqref{eq4_25} by restriction to $F\times F$ has the following form: Figure \ref{fig5}.
\begin{figure}[ht]
\begin{tabular}{|c|c|c|c|c|c|c|c|c|c|c|c|c|c|c|}
\hline
 &0&1&00&01&10&11&000&001&010&011&100&101&110&111\\
 \hline
 0&1&0&1&1& 0& 0& 1& 1& 1& 1& 0& 0& 0& 0\\
 \hline
 1& 0& 1& 0& 0& 1& 1& 0& 0& 0& 0& 1& 1& 1& 1\\
 \hline
 00&1& 0& 2& 1& 0& 0& 2& 2& 1& 1& 0& 0& 0& 0\\
 \hline
 01&1 &0& 1& 2& 0& 0& 1& 1& 2& 2& 0& 0& 0& 0\\
 \hline
 10&0 &1& 0& 0& 2& 1& 0& 0& 0& 0& 2& 2& 1& 1\\
 \hline
 11&0 &1& 0& 0& 1& 2& 0& 0& 0& 0& 1& 1& 1& 2\\
 \hline
000&1 &0& 2& 1& 0& 0& 3& 2& 1& 1& 0& 0& 0& 0\\
\hline
001&1 &0& 2& 1& 0& 0& 2& 3& 1& 1& 0& 0& 0& 0\\
\hline
010&1 &0 &1& 2& 0& 0& 1& 1& 3& 2& 0& 0& 0& 0\\
\hline
011& 1&0 &1& 2& 0& 0& 1& 1& 2& 3& 0& 0& 0& 0\\
\hline
100& 0&1 &0& 0& 2& 1& 0& 0& 0& 0& 3& 2& 1& 1\\
\hline
101&0 &1 &0& 0& 2& 1& 0& 0& 0& 0& 2& 3& 1& 1\\
\hline
110&0 &1 &0& 0& 1& 2& 0& 0& 0& 0& 1& 1& 3& 2\\
\hline
111&0 &1 &0& 0& 1& 2& 0& 0& 0& 0& 1& 1& 2& 3\\
\hline
\end{tabular}
\caption{$M_F$ for $F=\{0,1,00,\dots,111\}$}\label{fig5}
\end{figure}
\end{example}
\begin{remark}\label{rem4_10}
The spectral theory of $M_F$ appears to be difficult in general, but if 
$$F=F_k=\{x\in G^0\,|\, l(x)=k\}=\{x\,|\,x=(\omega_1\dots\omega_k),\omega_i\in\{0,1\}\}$$
then $M_k:=M_F$ may be generated recursively. 

Let $A=(a_{i,j})$ be an $n\times n$ matrix and set $\tau(A)_{i,j}:=a_{i,j}+1$. Then
$$M_{k+1}=\left(\begin{array}{cc}
\tau(M_k)&0\\
0&\tau(M_k)\end{array}\right).$$
\begin{equation}
M_1=\left(\begin{array}{cc}
1&0\\
0&1\end{array}\right), M_2=\left(\begin{array}{cccc}
2&1&0&0\\
1&2&0&0\\
0&0&2&1\\
0&0&1&2
\end{array}\right),\mbox{ etc; see Figure \ref{fig5}.}
\label{eqnew1}
\end{equation}
Set $\Lambda_k=\mbox{spectrum}_{l^2}(M_k)$; then
\begin{equation}
	\min\Lambda_k=1,\mbox{ and }\max\Lambda_k=2^k-1.
	\label{eqr4_10_1}
\end{equation}
\end{remark}

\begin{observation}\label{obs4_18}
Denoting the vertices in $G^0$ as in Figure \ref{fig3}, we get the following relations for the two systems of vectors $\{\delta_x\,|\,x\in G^0\}$ and $\{v_x\,|\, x\in G^0\}$:
\begin{equation}
	\delta_\ty=-v_0-v_1,\quad\|\delta_\ty\|_{H_E}^2=2,
	\label{eq418_1}
\end{equation}
and
\begin{equation}
	\|P_{F_k}\delta_\ty\|_{H_E}^2=\frac{2^k}{2^k-1}(\rightarrow 1).
	\label{eq418_2}
\end{equation}
\begin{proof}
From \eqref{eqnew1}, we see that $\lambda_k:=\max\Lambda_{F_k}=2^k-1$ has multiplicity 2 for all $k\in\bn$. We may pick two normalized eigenvectors $\xi_\lambda^{F_k}\in l^2(F_k)$:
$$\xi_{\lambda_+}^{F_k}=\frac{1}{\sqrt{2^{k-1}}}\left(\begin{array}{c}
1\\
1\\
\vdots\\
1\\
0\\
0\\
\vdots\\
0\end{array}\right),\mbox{ and }\xi_{\lambda_-}^{F_k}=\frac{1}{\sqrt{2^{k-1}}}\left(\begin{array}{c}
0\\
0\\
\vdots\\
0\\
1\\
1\\
\vdots\\
1\end{array}\right)$$
The other eigenvectors $\xi_\lambda^{F_k}$ in $l^2(F_k)$ corresponding to $\lambda<2^k-1$ satisfy $\ip{\xi_\lambda}{\chi_{F_k}}_{l^2}=0$.

By Theorem \ref{th4_4} and the observations following Corollary \ref{cor4_5} we get
$$\|P_{F_k}\delta_\ty\|_{H_E}^2=\sum_{\lambda\in\Lambda_{F_k}}|\ip{u_\lambda^{F_k}}{\delta_\ty}|^2=\frac{2^k}{2^k-1}.$$
Indeed by \eqref{eq4_13}, we have
$$P_{F_k}\delta_\ty=-\sqrt{\frac{2^{k-1}}{2^k-1}}(u_{\lambda_+}^{F_k}+u_{\lambda_-}^{F_k});$$
see \eqref{eq4_11}, and \eqref{eq418_2} follows.\end{proof}
\end{observation}

\section{The truncated operators $P_F\Delta P_F$}\label{trun}

In the general framework of section \ref{hilb} and \ref{grap} we introduced symmetric Hilbert spaces $\H$ and associated operators $\Delta$. We proved (Theorem \ref{th3_5}) that the setup includes the most general class of graph Laplacians for weighted graphs $(G,\mu)$. In the latter case, the symmetric Hilbert space is $\H=H_E=$ the energy Hilbert space of Definition \ref{def3_2}. In all cases, we show that the Hilbert space under consideration is associated with a positive definite function
\begin{equation}
	M(x,y)=\ip{v_x}{v_y}_{\H}
	\label{eq5_1}
\end{equation}
where $\{v_x\,|\,x\in X\}$ is a system of vectors in $\H$, and $\H=\cj{\spn}\{v_x\,|\,x\in X\}$; see Theorem \ref{th3_10}. Further we show that it is possible to choose the family $(v_x)_{x\in X}$ such that each $v_x$ is in the domain of $\Delta$, i.e., $v_x\in\dom(\Delta)$ for all $x\in X$; see Theorem \ref{th3_5} and Lemma \ref{lem3_7}.

In section \ref{diag}, we reduced the study of operators $\Delta$ in $\H$ to its finite truncations. Specifically, for each finite subset $F\subset X$, we introduced in Theorem \ref{th4_4} the orthogonal projection $P_F$ onto 
\begin{equation}
	\H(F):=\spn_{\H}\{v_x\,|\,x\in F\}.
	\label{eq5_2}
\end{equation}
When $F$ is given, let $\{u_\lambda^F\,|\,\lambda\in\Lambda_F\}$ be the ONB in $\H(F)$ introduced in \eqref{eq4_11}. Then with Dirac's notation, we have
\begin{equation}
	P_F=\sum_{\lambda\in\Lambda_F}|u_\lambda^F\rangle\langle u_\lambda^F|.
	\label{eq5_3}
\end{equation}
It follows from \eqref{eq4_11} that each $u_\lambda^F$ is in $\dom(\Delta)$; and  as a result that the finite-rank truncations
\begin{equation}
	P_F\Delta P_F
	\label{eq5_4}
\end{equation}
are well defined. For fixed $F$, the matrix with respect to the ONB $\{u_\lambda^F\,|\,\lambda\in\Lambda_F\}$ is
\begin{equation}
	\ip{u_\lambda^F}{\Delta u_{\lambda'}^F}_{\H}
	\label{eq5_5}
\end{equation}
where $\lambda$ is a row-index, and $\lambda'$ a column index.

The purpose of this section is to approximate $\Delta$ with its finite truncations $P_F\Delta P_F$. To do this use some chosen nested system
\begin{equation}
	F_1\subset F_2\subset\dots F_k\subset\dots\subset X
	\label{eq5_6}
\end{equation}
such that
\begin{equation}
	\cup_{k\in\bn}F_k=X.
	\label{eq5_7}
\end{equation}
With that choice
\begin{equation}
	\lim_{k\rightarrow\infty}P_{F_k}=I_{\H},
	\label{eq5_8}
\end{equation}
and we wish to study the corresponding limit
\begin{equation}
	\lim_{k\rightarrow\infty}P_{F_k}\Delta P_{F_k}.
	\label{eq5_9}
\end{equation}
\subsection{Graph applications}\label{appl}
In view of Lemma \ref{lem3_7}, it is practical to select a base point $0\in G^0$, and for each $x\in G^0$, choose the unique solution $v_x\in H_E$ to 
\begin{equation}
	\ip{v_x}{u}_E=u(x)-u(0),\quad(u\in H_E).
	\label{eq5_10}
\end{equation}
Recall (Theorem \ref{th3_5}), in this case
\begin{equation}
	\Delta v_x=\delta_x-\delta_0,\quad(x\in G^0\setminus\{0\}.
	\label{eq5_11}
\end{equation}
Before turning to the approximation \eqref{eq5_9}, we prove the following
\begin{lemma}\label{lem5_1}
Let $(G,\mu), \Delta,0\in G^0, H_E,\{v_x\,|\,x\in G^0\setminus\{0\}\}$ be as described above, and let $F\subset G^0\setminus\{0\}$ be a finite subset. Then
\begin{equation}
	P_F\delta_0=-\sum_{\lambda\in\Lambda_F}\frac{1}{\sqrt{\lambda}}\ip{\xi_\lambda^F}{\chi_F}_{l^2}u_\lambda^F;
	\label{eq5_12}
\end{equation}
\begin{equation}
	\|P_F\delta_0\|_{H_E}^2=\sum_{\lambda\in\Lambda_F}\frac{|\ip{\xi_\lambda^F}{\chi_F}_{l^2}|^2}{\lambda};
	\label{eq5_13}
\end{equation}
and
\begin{equation}
	P_F\Delta P_F\mbox{ is a rank-1 perturbation of the diagonal operator }
	\label{eq5_14}
	\end{equation}
\begin{equation}	
	D_F=\left(\begin{array}{cccc}
	\lambda_1^{-1}&0&\cdots&0\\
	0&\lambda_2^{-1}&\cdots&0\\
	\vdots&\vdots&\ddots&\vdots\\
	0&0&\cdots&\lambda_{n_F}^{-1}
	\end{array}\right)
	\label{eq5_15}
\end{equation}
where $\Lambda_F=\{\lambda_1,\lambda_2,\dots,\lambda_{n_F}\}$ with eigenvalues counted with repetition according to multiplicity.
\end{lemma}

\begin{proof}
Ad \eqref{eq5_12}
$$P_F\delta_0=\sum_{\lambda\in\Lambda_F}\ip{u_\lambda^F}{\delta_0}u_\lambda^F=\sum_{\lambda\in\Lambda_F}\frac{1}{\sqrt{\lambda}}\sum_{x\in F}\cj{\xi_\lambda^F(x)}\ip{v_x}{\delta_0}_Eu_\lambda^F=$$
$$\stackrel{\mbox{ by \eqref{eq5_10}}}{=}-\sum_{\lambda\in\Lambda_F}\frac{1}{\sqrt{\lambda}}\sum_{x\in F}\cj{\xi_\lambda^F(x)}u_\lambda^F=-\sum_{\lambda\in\Lambda_F}\frac{1}{\sqrt{\lambda}}\ip{\xi_\lambda^F}{\chi_F}_{l^2}u_\lambda^F,$$
which is \eqref{eq5_12}. Note that \eqref{eq5_13} is immediate from this by Parseval.

Ad \eqref{eq5_14}. We compute the matrix representation \eqref{eq5_5}
$$\ip{u_\lambda^F}{\Delta u_{\lambda'}^F}_E\stackrel{\mbox{by \eqref{eq4_11}}}{=}\frac{1}{\sqrt{\lambda\lambda'}}\sum_{x,y\in F}\cj{\xi_\lambda^F(x)}\xi_{\lambda'}^F(y)\ip{v_x}{\Delta v_y}_E\stackrel{\mbox{by \eqref{eq5_11}}}{=}$$
$$\frac{1}{\sqrt{\lambda\lambda'}}\sum_{x,y\in F}\cj{\xi_\lambda^F(x)}\xi_{\lambda'}^F(y)\ip{v_x}{\delta_y-\delta_0}_E=
\frac{1}{\sqrt{\lambda\lambda'}}\sum_{x,y\in F}\cj{\xi_\lambda^F(x)}\xi_{\lambda'}^F(y)(\delta_{x,y}+1)=$$
$$
\frac{1}{\sqrt{\lambda\lambda'}}\left(\left(\sum_{x\in F}\cj{\xi_\lambda^F(x)}\xi_{\lambda'}^F(x)\right)+\ip{\xi_\lambda^F}{\chi_F}_{l^2}\ip{\chi_F}{\xi_{\lambda'}^F}_{l^2}\right)
\stackrel{\mbox{by \eqref{eq5_12}}}{=}\frac{1}{\lambda}\delta_{\lambda,\lambda'}+\ip{u_\lambda^F}{P_F\delta_0}\ip{P_F\delta_0}{u_{\lambda'}^F}$$
which is the $\lambda,\lambda'$-coefficient of the operator
$$
\left(\begin{array}{cccc}
	\lambda_1^{-1}&0&\cdots&0\\
	0&\lambda_2^{-1}&\cdots&0\\
	\vdots&\vdots&\ddots&\vdots\\
	0&0&\cdots&\lambda_{n_F}^{-1}
	\end{array}\right)+|P_F\delta_0\rangle\langle P_F\delta_0|$$
\end{proof}
\begin{remark}\label{rem5_2}
The function in \eqref{eqA6},
\begin{equation}
	M(x_1,x_2)=x_1\wedge x_2,\quad x_1,x_2\in[0,1)
	\label{eqr1}
\end{equation} is the continuous analogue of the discrete version \eqref{eq4_25}. And \eqref{eq4_25} in turn is a special case of \eqref{eq4_1}, i.e.,
\begin{equation}
	M(x,y)=\ip{v_x}{v_y}_{\H}
	\label{eqr2}
\end{equation}
valid for the most general Hilbert space $\H$.

The purpose of Lemma \ref{lem5_1} is to obtain a discrete version of a Green's function for $\Delta$. To see how Lemma \ref{lem5_1} compares to the classical case, Example \ref{ex1_1.5}, note that if $\varphi\in C^2[0,1]$ and $\varphi'(1)=0$ then
\begin{equation}
	\int_0^1\varphi''(y)(y\wedge x)\,dy=\varphi(0)-\varphi(x).
	\label{eqr3}
\end{equation}
As is known, the function in \eqref{eqr1} is the Green's functions for $\Delta=-\frac{d^2}{dx^2}$; and we think equation \eqref{eqr3} as a continuous variant of our formula \eqref{eq3_17}-\eqref{eq3_18} in Lemma \ref{lem3_7}. 

The idea is that if $\H=H_E$ from a graph Laplacian $\Delta$ of a weighted graph $(G,\mu)$, then the function $M(\cdot,\cdot)$ in \eqref{eqr2} is the Green's function for $\Delta$.
\end{remark}

\subsection{Boundedness}\label{boun}

 In this subsection we study an intriguing interrelationship between the family of matrices $M_F$ on the one hand, and the Laplace operator $\Delta$ on the other. The operator $\Delta$  will be considered in the energy Hilbert space $\H_E$. While boundedness may be easily discerned when $\Delta$ is viewed as an operator in $l^2$, this is not the case when the ambient Hilbert space is $\H_E$. The result below is the assertion that boundedness is equivalent with the presence of a spectral gap for the system of matrices $M_F$. Note that in Example \ref{ex4_7}, the matrix $M_F$ encodes agreement in the comparison of finite words (a Google matrix), and the result therefore yields spectral data for the Google matrix as a consequence of operator theory of $\Delta$.

 The information carried in Lemma \ref{lem5_1} suggests a ``spectral reciprocity''. For each finite subset $F\subset G^0\setminus\{0\}$, the operator
 
$$D_F=\left(\begin{array}{cccc}
	\lambda_1^{-1}&0&\cdots&0\\
	0&\lambda_2^{-1}&\cdots&0\\
	\vdots&\vdots&\ddots&\vdots\\
	0&0&\cdots&\lambda_{n_F}^{-1}
	\end{array}\right)$$
 encodes the numbers $\{\lambda^{-1}\,|\,\lambda\in\Lambda_F\}$. We proved the formula
 \begin{equation}
	P_F\Delta P_F=D_F+|P_F\delta_0\rangle\langle P_F\delta_0|,
	\label{eq5_16}
\end{equation}
where $P_F\Delta P_F$ is a ``matrix-corner'' of the infinite dimensional operator $\Delta$. Specifically,
$P_F\Delta P_F$ arises from $\Delta$ as
\begin{equation}
	\Delta=\left(\begin{array}{cc}
	P_F\Delta P_F& P_F\Delta P_F^\perp\\
	P_F^\perp\Delta P_F& P_F^\perp\Delta P_F^\perp
	\end{array}\right)
	\label{eq5_17}
\end{equation}
where $P_F^\perp:=I_\H-P_F$ is the projection onto the orthocomplement
\begin{equation}
	\H(F)^\perp:=\H\ominus\H(F)=\{u\in\H\,|\,\ip{u}{v}_\H=0\mbox{ for all }v\in\H(F)\}.
	\label{eq5_18}
\end{equation}
The last term in \eqref{eq5_16} is a rank-1 operator, i.e., $R_F=|u_F\rangle\langle u_F|$, $u_F:=P_F\delta_0$. Equivalently 
\begin{equation}
	R_F=\|u_F\|_\H^2 P_{u_F}
	\label{eq5_19}
\end{equation}
where $P_{u_F}=$ the projection onto $\bc u_F=$ the 1-dimensional space spanned by $u_F$.

Hence, for the operator norm $R_F:\H\rightarrow\H$, we have
\begin{equation}
	\|R_F\|_{\H\rightarrow\H}=\sup\{\|R_Fu\|_\H\,|\, \|u\|_\H=1\}=\|u_F\|_\H^2.
	\label{eq5_20}
\end{equation}
Setting $\V:=\spn\{v_x\,|\,x\in G^0\setminus\{0\}\}=$ all finite linear combinations, we get
\begin{equation}
	\lim_{F\rightarrow\infty} P_F\delta_0=\delta_0
	\label{eq5_21}
\end{equation}
and
\begin{equation}
	\lim_{F\rightarrow\infty}|u_F\rangle\langle u_F|=|\delta_0\rangle\langle\delta_0|.
	\label{eq5_22}
\end{equation}

\begin{corollary}\label{cor5_2}
Let $F_1\subset F_2\subset\dots$ be an ascending family of finite sets satisfying \eqref{eq5_6}-\eqref{eq5_7}. It follows that
\begin{equation}
	\lim_{k\rightarrow\infty} D_{F_k}=\Delta+\|\delta_0\|_\H^2 P_{\delta_0}\mbox{ on }\V.
	\label{eq5_23}
\end{equation}
\end{corollary}
\begin{proof}
By \eqref{eq5_23} we mean that the limit 
\begin{equation}
	\lim_{k\rightarrow\infty}\ip{u}{D_{F_k}v}_\H=\ip{u}{\Delta v}_\H+\ip{u}{\delta_0}\ip{\delta_0}{v}
	\label{eq5_24}
\end{equation}
is valid for all $u,v\in\V$. 

But this conclusion is contained in Lemma \ref{lem5_1} and the discussion above.

\end{proof}

\begin{corollary}\label{cor5_3}
Let $(G,\mu), 0\in G^0,\Delta$ and $\H=H_E$ be as above. Set
\begin{equation}
	M(x,y)=\ip{v_x}{v_y}_\H,\quad x,y\in G^0\setminus\{0\},
	\label{eq5_25}
\end{equation}
and
\begin{equation}
M_F:=M|_{F\times F}.	
	\label{eq5_26}
\end{equation} 
 Then 
\begin{equation}
	\delta_\Delta:=\inf_{F\in\F}\min\{\lambda\in\Lambda_F\}>0
	\label{eq5_27}
\end{equation}
if and only if $\Delta$ is a bounded operator $H_E\rightarrow H_E$.
\end{corollary}

\begin{proof}
Suppose $\delta_\Delta>0$. Let $F\in\F$, and let $u\in \H(F)$. Then by Lemma \ref{lem5_1}
$$\ip{u}{\Delta u}_{\H}=\ip{u}{D_F u}_\H+|\ip{u}{P_F\delta_0}_\H|^2\leq\delta_\Delta^{-1}\|u\|_\H^2+\|u\|_\H^2\|P_F\delta_0\|_\H^2\leq(\delta_\Delta^{-1}+\|\delta_0\|_\H^2)\|u\|_\H^2.$$
Since $\Delta$ is Hermitian, this implies boundedness; and 
\begin{equation}
	\|\Delta\|_{\H\rightarrow\H}=\sup_{u\in\V,\|u\|_\H=1}\ip{u}{\Delta u}\leq\delta_\Delta^{-1}+\|\delta_0\|_\H^2.
	\label{eq5_28}
\end{equation}
Note that \eqref{eq5_28} yields an {\it a priori} bound on the norm of $\Delta$.

Conversely, suppose $\Delta$ is a bounded operator. Since the limit \eqref{eq5_24} exists, and $$\|D_F\|_{\H\rightarrow\H}=\max_{\lambda\in\Lambda_F}\{\lambda^{-1}\}=\frac{1}{\min_{\lambda\in\Lambda_F}\{\lambda\}},$$
we have $\delta_\Delta>(\|\Delta\|_{\H\rightarrow\H})^{-1}$ and the conclusion follows.
\end{proof}

\subsection{Application}\label{Appl}
In Example \ref{ex4_7} we introduced the matrix \eqref{eq4_25}
$M(x,y):=\#(\P(x)\cap \P(y))$, as a measure of agreement of sets of words represented by the paths to $x$ as compared to $y$. 

One may compute the spectrum of $M_F$ for all finite subsets $F\subset G^0\setminus\{\ty\}$, but it is difficult to directly compute the spectral gap number $\delta_\Delta$ in \eqref{eq5_27} for this example. 

Hence as an application of our spectral representation of $\Delta$ in $l^2(G)$, or in $H_E$ from \cite{DuJo08}, we can show that $\|\Delta\|_{\H_E\rightarrow H_E}<\infty$. 

\begin{theorem}\label{th5_4}\cite[Theorem 3.26]{DuJo08}
Let $\mu_c$ be the semicircular measure on $[-1,1]$ 
$$d\mu_c=\frac{2}{\pi}\sqrt{1-x^2}\,dx,$$
and let $\mu_{c+p}$ be the measure on $[-1,1]$ given by 
$$d\mu_{c+p}=\frac{\frac2{\pi}\sqrt{1-x^2}}{\frac32-\sqrt{2} x}\,dx.$$
Let  $M_{c+p}$ be the operator of multiplication by $3-2\sqrt{2}x$ on $L^2(\mu_{c+p})$, and $M_c$ the operator of multiplication by $3-2\sqrt{2}x$ on $L^2(\mu_c)$.
Then the Laplacian $\Delta:l^2\rightarrow l^2$ is unitarily equivalent to the multiplication operator 
\begin{equation}
M_{c+p}\oplus\oplus_{n=1}^\infty M_c\mbox{ on } L^2(\mu_{c+p})\oplus\oplus_{n=1}^\infty L^2(\mu_c).	
	\label{eq5_29}
\end{equation}

\end{theorem}

 So, as an application of Corollary 5.3, we get the following
\begin{corollary}\label{cor5_4}
Let $(G,\mu)=(\mbox{tree},1)$ be the graph in Example \ref{ex4_7} and let $\delta_\Delta$ be the number \eqref{eq5_27} for this example.  Then $\delta_\Delta>0$; i.e., there is a spectral gap in the Google matrix; see Figure \ref{fig5}.
\end{corollary}

\begin{proof}
In view of Corollary \ref{cor5_3}, we must show that $\Delta$ from Example \ref{ex4_7}, i.e., $(G,\mu)=(\mbox{tree},1)$ is bounded in $H_E$; that $\Delta:H_E\rightarrow H_E$ is a bounded operator. The boundedness of $\Delta:l^2\rightarrow l^2$ is contained in the spectral representation \eqref{eq5_29} in Theorem \ref{th5_4}. Indeed there is unitary equivalence $W:l^2\rightarrow L^2([-1,1],\nu,\K)$ where $\nu$ is the spectral measure and $\K\approx l^2$ is the Hilbert space which accounts for multiplicity; and $W$ satisfies
\begin{equation}
	W\Delta=(3-2\sqrt2x)W.
	\label{eq5_30}
\end{equation}
It follows that the quadratic form 
\begin{equation}
	\psi\mapsto\int_{-1}^1(3-2\sqrt2 x)\|\psi(x)\|^2_{\K}\,d\nu(x)=:\|\psi\|_{E,L^2(\nu)}^2
	\label{eq5_31}
\end{equation}
extends $u\mapsto\|u\|_{H_E}^2$.
Setting $Wu=\hat u=\psi$, we get 
$$\|\Delta u\|_{H_E}^2=\int_{-1}^1(3-2\sqrt{2}x)\|(3-2\sqrt{2}x)\hat u(\cdot)\|_\K^2\,d\nu(x)\leq 3^2\int_{-1}^1(3-2\sqrt2x)\|\hat u(x)\|_\K^2\,d\nu(x)=
9\|\hat u\|_{E,L^2(\nu)}^2=9\|u\|_{H_E}^2.$$
Hence, $\|\Delta\|_{\H_E\rightarrow H_E}\leq 3$.
\end{proof}
\begin{corollary}\label{cor5_7}
Let $(G,\mu),0\in G^0,\Delta,$ and $\H=H_E$ be as in Corollary \ref{cor5_3}. As an operator $\Delta:H_E\rightarrow H_E$, there is a bounded inverse if and only if 
\begin{equation}
	\sigma=\sup_{F\in\F}\max\{\lambda\in\Lambda_F\}<\infty.
	\label{eq5_35}
\end{equation}
\end{corollary}
\begin{proof}
By Lemma \ref{lem5_1}, invertibility of $\Delta:H_E\rightarrow H_E$ is decided by the presence of a global {\it a priori} bound on the operators $(P_F\Delta P_F)^{-1}$ as $F$ ranges over $\F$. If $\sigma$ in \eqref{eq5_35} is finite, then
$$\sup_F\|\Delta_F^{-1}\|_{H_E\rightarrow H_E}<\sigma^{-1}$$
where $\Delta_F:=P_F\Delta P_F$. And conversely. 
\end{proof}

\begin{corollary}\label{cor5_8}
The operator $\Delta$ in Example \ref{ex4_7} (i.e., $(G,\mu)=(\mbox{tree},1)$) and Corollary \ref{cor5_4} does not have a bounded inverse.
\end{corollary}

\begin{proof}
Follows from \eqref{eqr4_10_1} in Remark \ref{rem4_10}.
\end{proof}

\section{The Green's function}\label{gree}
 In this section we prove that the semidefinite functions introduced in section \ref{Appl} serve as Green's functions for graph Laplacians.

We treat the general case of the function 
$$G^0\times G^0\ni(x,y)\mapsto M(x,y)=\ip{v_x}{v_y}_E$$
from \eqref{eq5_25}, and we show that it is analogous to the standard Green's function for $\Delta$ in the continuous case; see Example \ref{ex1_1.5} and Remark \ref{rem5_2}

\begin{theorem}\label{th6_1}
Consider the function $M(\cdot,x)$ associated with a fixed weighted graph $(G,\mu)$ with graph Laplacian $\Delta$. The action of $\Delta$ on this function will be denoted $\Delta_\cdot M(\cdot,x)$ where the dot represents the action variable. Then 
$$-(\Delta_\cdot M(\cdot,x))(y)=\delta_{x,y}+1-\mu(y)M(y,x).$$
\end{theorem}
\begin{remark}\label{rem6_2}
We have restricted the variables $x,y$ to $G^0\setminus\{0\}$ where $0$ is a chosen fixed base point in $G^0$, and where we make the convention
$v_x(0)=0$, for all $x\in G^0\setminus\{0\}$. 
\end{remark}
\begin{proof}[Proof of Theorem \ref{th6_1}]
For $x,y\in G^0\setminus\{0\}$ we have
$$-(\Delta_\cdot M(\cdot,x))(y)\stackrel{\mbox{by \eqref{eq3_5}}}{=}\sum_{z\sim y}\mu_{yz}(M(z,x)-M(y,x))\stackrel{\mbox{by \eqref{eq3_2}}}{=}\sum_{z\sim y}\mu_{yz}\ip{v_z}{v_x}_E-\mu(y)M(y,x)=$$
$$\stackrel{\mbox{by \eqref{eq3_18}}}{=}\ip{\Delta v_y}{v_x}_E-\mu(y)M(y,x)\stackrel{\mbox{by \eqref{eq3_17}}}{=}\ip{\delta_y-\delta_0}{v_x}_E-\mu(y)M(y,x)=\delta_{xy}+1-\mu(y)M(y,x).$$
\end{proof}

In the third step of the computation we used the following lemma.

\begin{lemma}\label{lem6_3} For points in $G^0\setminus\{0\}$, we have the following identity
\begin{equation}
	\sum_{z\sim x}\mu_{xz}(v_x-v_z)=\Delta v_x.
	\label{eq6_1}
\end{equation}
\end{lemma}
\begin{proof}
Let $y\in G^0\setminus\{0\}$ then 
$$\ip{v_y}{\sum_{z\sim x}\mu_{xz}(v_x-v_z)}_E=\mu(x)v_y(x)-\sum_{z\sim x}\mu_{xz}v_y(z)=$$$$(\Delta v_y)(x)=(\delta_y-\delta_0)(x)=\delta_{xy}=(\Delta v_x)(y)=\ip{v_y}{\Delta v_x}_E.$$
Since $\spn\{v_y\,|\,y\in G^0\setminus\{0\}\}$ is dense in $H_E$, the desired conclusion \eqref{eq6_1} follows.
\end{proof}

\bibliographystyle{alpha}
\bibliography{sdcuo}

\end{document}